\documentclass[12pt]{amsart}
\usepackage{amssymb, enumerate, mdwlist}

\evensidemargin\paperwidth
\advance\evensidemargin-\textwidth
\advance\oddsidemargin-1in
\evensidemargin\oddsidemargin

\topmargin\paperheight
\advance\topmargin-\textheight
\advance\topmargin-1in

\theoremstyle{plain}
\newtheorem{theorem}{Theorem}[section]
\theoremstyle{plain}
\newtheorem{proposition}[theorem]{Proposition}
\theoremstyle{plain}
\newtheorem{lemma}[theorem]{Lemma}
\theoremstyle{plain}

\theoremstyle{plain}

\theoremstyle{plain}
\newtheorem{definition}[theorem]{Definition}
\theoremstyle{plain}

\theoremstyle{remark}

\theoremstyle{remark}
\newtheorem{example}[theorem]{Example}
\theoremstyle{remark}

\newcommand{\B}{\mathcal{B}}
\newcommand{\C}{\mathbb{C}}
\newcommand{\Cstar}{\mathrm{C}^*}
\newcommand{\CstarU}{\mathrm{C}^*_\mathrm{u}}

\newcommand{\Hil}{\mathcal{H}}

\newcommand{\N}{\mathbb{N}}
\newcommand{\op}{\mathrm{op}}

\newcommand{\R}{\mathbb{R}}

\title
[Uniform Roe algebras]
{Finite-dimensional approximation properties for uniform Roe algebras}
\author{Hiroki Sako}
\address
{Faculty of engineering, Niigata University, Niigata 950-2181, Japan}
\email
{sako@eng.niigata-u.ac.jp}

\begin{document}

\begin{abstract}
We study property A for metric spaces $X$ with bounded geometry introduced by Guoliang Yu.
Property A is an amenability-type condition, which is less restrictive than amenability for groups.
The property has a connection with finite-dimensional approximation properties in the theory of operator algebras. It has been already known that property A of a metric space $X$ with bounded geometry is equivalent to
nuclearity of 
the uniform Roe algebra $\CstarU(X)$.
We prove that exactness and local reflexivity of $\CstarU(X)$
also characterize property A of $X$.
\end{abstract}


\maketitle

\section{Introduction}
The notion of amenability connects functional analytic aspects of groups to geometrical aspects of them.
It is characterized by the existence of F{\o}lner sets, which are finite subsets with relatively small boundaries (\cite{Folner}).
In the theory of operator algebras, 
there are a number of characterizations of amenability.
A typical example is Lance's theorem. He proved that
a discrete group $G$ is amenable if and only if the reduced
group C$^*$-algebra $\Cstar_\mathrm{red}(G)$ is nuclear
(\cite[Theorem 4.2]{Lance}). 

Property A introduced by Guoliang Yu is an amenability-type condition for discrete metric spaces.
Yu proved that property A is a sufficient condition for the
coarse Baum--Connes conjecture (\cite[Theroem 1.1 and Theorem 2.7]{Yu:CoarseHilbert}).
An operator algebraic characterization was given by Skandalis, Tu and Yu.
Property A of a metric space $X$ with bounded geometry is equivalent to nuclearity of the uniform Roe algebra $\CstarU(X)$ (\cite[Theorem 5.3]{SkandalisTuYu}).

The uniform Roe algebra
$\CstarU(X)$ is a natural linear representation of $X$.
It is also called the translation C$^*$-algebra based on the Hilbert space $\ell_2 (X)$ (Roe's lecture note \cite[section 4.4]{RoeLectureNote}).
In the theory of C$^*$-algebras,
nuclearity, exactness, and local reflexivity are important.
The purpose of this paper is to show these properties are equivalent for the algebra $\CstarU(X)$.

\begin{theorem}\label{Theorem Main Intro}
For a metric space $X$ with bounded geometry,
the following conditions are equivalent:
\begin{enumerate}
\item\label{Theorem Intro A of Yu}
The space $X$ has property A.
\item\label{Theorem Intro Nuclear}
The uniform Roe algebra $\CstarU(X)$ is nuclear.
\item\label{Theorem Intro Exact}
The algebra $\CstarU(X)$ is exact.
\item\label{Theorem Intro Local Reflexive}
The algebra $\CstarU(X)$ is locally reflexive.
\item\label{item: ONLP}
The space $X$ has the operator norm localization property in \cite{ONLPoriginal}.
\end{enumerate}
\end{theorem}
The equivalence between property A and nuclearity of $\CstarU(X)$ was proved by Skandalis, Tu and Yu.
By the theory of  C$^*$-algebra, nuclearity implies exactness, and exactness implies local reflexivity.
It is proved in \cite{SakoONLP} that the operator norm localization property is equivalent to property A.
We will deduce the operator norm localization property of $X$ from local reflexivity of $\CstarU(X)$ in Theorem \ref{theorem: main}.

The case of discrete groups has been studied by many authors.
For a countable discrete group $G$, we can define a metric function which makes $G$ a metric space with bounded geometry. 
The uniform Roe algebra of $G$ is isomorphic to $\ell_\infty(G) \rtimes_\mathrm{red} G$.
The notion of an exact group was introduced by  Kirchberg and Wassermann. 
It is characterized by exactness of the reduced group C$^*$-algebra $\Cstar_\mathrm{red} (G)$ (\cite[Theorem 5.2]{KirchbergWassermann}). 
The class of exact groups is much larger than that of amenable groups.
Ozawa proved that exactness of $\Cstar_\mathrm{red} (G)$ implies nuclearity of the uniform Roe algebra $\CstarU(G)$ (\cite[Theorem 3]{OzawaExactness}). 
As a corollary, exactness and nuclearity are equivalent for
the uniform Roe algebra $\CstarU(G)$. 
Nuclearity of $\CstarU(G)$ is equivalent to property A of the metric space $G$ (Higson and Roe \cite[Theorem 1.1]{HigsonRoe}, Anantharaman-Delaroche and Renault \cite{ADR}).
It follows that exactness of the discrete group $G$ is equivalent to property A of the metric space $G$.

For a general metric space $X$ with bounded geometry, 
equivalence between property A of $X$ and exactness of $\CstarU(X)$ had been an open problem, until the draft of this paper was announced.
In \cite[Corollary 30]{PaperBrodzkiNibloWright}, Brodzki, Niblo and Wright proved it in the case that the space $X$ uniformly embeds into a discrete group.

\section{Preliminaries}

\subsection{Notations}
Let $(X, d)$ be a metric space.
For $x, y \in X$, the distance $d(x, y)$ can be $\infty$.
For $x \in X$ and a positive number $R$, denote by $N(x; R)$ the ball 
$\{y \in X \ |\ d(x, y) \le R\}$.
We assume that for every positive number $R$, $\sup_{x \in X} \sharp N(x; R) < \infty$.
If this condition holds, then we say that $X$ has {\it bounded geometry}.
The relation $\{(x, y) \in X^2 \ |\ d(x, y) < \infty \}$ is an equivalence relation on $X$.
An equivalence class with respect to this relation is called a {\it coarse connected component}.
Note that for a subset $Y \subset X$, if the values of $d |_{Y \times Y}$ are real numbers,
then $Y$ is contained in a coarse connected component of $X$.
A subset $C \subset X^2$ is said to be {\it controlled}, if $\sup \{d(x, y) \ |\ (x, y) \in C\} < \infty$.

\subsection{Property A}
Amenability for a discrete group is characterized by
the existence of F{\o}lner sets, which are finite subsets with relatively small boundaries.
A simple replacement of the F{\o}lner condition
does not provide an appropriate notion of amenability for metric spaces.
It requires a family of F{\o}lner sets.

\begin{definition}[Definition 2.1 in \cite{Yu:CoarseHilbert}]
A discrete metric space $(X, d)$ 
is said to have {\rm property A}
if for every positive number $\epsilon$ and every positive number $R$, 
there exist
a positive number $S$ 
and a family of finite subsets 
$(A_x)_{x \in X}$ of $X \times \N$ such that
\begin{enumerate}
\item
for every $x \in X$,
$A_x \subset N(x; S) \times \N$;
\item
for every $x \in X$,
$(x, 1) \in A_x$;
\item
for every $x, y \in X$,
if $d(x, y) < R$, then the symmetric difference $A_x \bigtriangleup A_y$ satisfies
$\sharp(A_x \bigtriangleup A_y) 
< \epsilon \sharp(A_x \cap A_y)$.
\end{enumerate}
\end{definition}

\begin{example}
\begin{itemize}
\item
Every discrete subgroup of the general linear group over a field has property A (Guentner, Higson, and Weinberger \cite{GuentnerHigsonWeinberger}).

\item
If two groups have property A, 
then an amalgamated free product has property A (Tu \cite[Corollary 9.5]{Tu}, Dykema \cite{Dykema}, Ozawa \cite[Corollary 2]{BoundaryAme}). 

\item
If a group $\Gamma$ is a hyperbolic relative to subgroups with property A, 
then $\Gamma$ has property A (Ozawa \cite[Corollary 3]{BoundaryAme}). 

\item
Asymptotic dimension of finitely generated groups was introduced by Gromov \cite[Section 1.E]{Gromov}.
This notion can be applied for metric spaces with bounded geometry.
If a metric space $X$ with bounded geometry has finite asymptotic dimension, then $X$ has property A
(Higson and Roe \cite[Lemma 4.3]{HigsonRoe}).
\end{itemize}
\end{example}

\subsection{Uniform Roe algebra}

Every bounded linear operator $a$ on $\ell_2(X)$ is uniquely determined by the family of complex numbers
\[[a_{x, y}]_{x, y \in X} = [\langle a \delta_y, \delta_x \rangle ]_{x, y \in X}.\]
We call $[a_{x, y}]_{x, y \in X}$ the matrix expression of $a$.
For a positive number $R$,
we say that the {\it propagation} of $a$ is at most $R$, if
for every $x, y \in X$,
\[d(x, y) > R \Rightarrow a_{x, y} = 0.\]
Denote by $E_R$ the space of all the operators on $\ell_2(X)$ whose propagation is at most $R$.
We call an element in $\cup_R E_R$ an operator with finite propagation.
The space $\cup_R E_R$ corresponds to the space of bounded functions on controlled sets.

\begin{lemma}[Conclusion of Lemma 4.27 in \cite{RoeLectureNote}]
\label{lemma: operator given by bounded controlled function}
For every positive number $R$ and 
for every bounded function $\zeta$ on $\left\{(x, y) \in X^2 \ |\ d(x, y) \le R \right\}$,
there exists an element $a$ of $E_R$ such that $[a_{x, y}]_{x, y \in X} = [\zeta(x, y) ]_{x, y \in X}$.
\end{lemma}

The uniform Roe algebra $\CstarU(X)$ is defined as the operator norm closure of $\cup_R E_R$. 
Since $\cup_R E_R$ is closed under multiplication and adjoint,
the uniform Roe algebra is a C$^*$-algebra.
Note that coarse connected components correspond to central projections in $\CstarU(X)$.

A coarse geometric property of $X$ sometimes implies an operator algebraic property of $\CstarU(X)$. The following is a typical example:
\begin{theorem}
[Theorem 5.3 in \cite{SkandalisTuYu}]
Let $X$ be a metric space with bounded geometry. The space $X$ has property A if and only if $\CstarU(X)$ is nuclear.
\end{theorem}

\subsection{Operator norm localization property}

Chen, Tessera, Wang, and Yu defined 
{\it the operator norm localization property}
in \cite[Section 2]{ONLPoriginal}. 
The original definition is given for a general metric space $X$.
For a metric space $X$ with bounded geometry, we use the following condition as the definition.
\begin{definition}
A metric space $X$ with bounded geometry is said to have {\rm the operator norm localization property}, if 
for every $0 < c < 1$ and every positive number $R$,
the following condition $(\alpha)$ holds:
there exists a positive number $S$ satisfying that
for every operator $a \in E_R$,
there exists a unit vector $\eta \in \ell_2(X)$ such that the diameter of
$\mathrm{supp}(\eta)$ is at most $S$ and
$c \| a \| \le \| a \eta \|$.
\end{definition}

The most important point is that $S$ is independent of the choice of $a \in E_R$.
This definition is a necessary and sufficient condition of the original definition.
See \cite[Proposition 3.1]{SakoONLP} for the proof.
We can replace the quantifier of $c$.

\begin{lemma}[Proposition 3.1 of \cite{SakoONLP}]
\label{Lemma quantifier}
A metric space $X$ with bounded geometry has the operator norm localization property if 
there exists $0 < c < 1$ such that for every positive number $R$,
condition $(\alpha)$ holds.
\end{lemma}

The following theorem is the main result of \cite{SakoONLP}:
\begin{theorem}[Theorem 4.1 of \cite{SakoONLP}]
For a metric space with bounded geometry, the operator norm localization property is equivalent to property A.
\end{theorem}

Let us pay attention on spaces without the operator norm localization property.
For a subset $Y$ of $X$, 
denote by $P_Y$ the orthogonal projection from $\ell_2(X)$ onto 
$\ell_2(Y)$.

\begin{lemma}\label{lemma: not ONLP}
If $(X, d)$ does {\rm not} have the operator norm localization property, then there exist
a sequence of disjoint subsets $V_n$ of $X$,
a sequence of positive matrices $b_n$ acting on $\ell_2(V_n)$ with norm $1$,
and a sequence $S_n$ of positive numbers
satisfying that
\begin{enumerate}
\item
for every $n \in \N$ and $Y \subset V_n$, if $\mathrm{diam}(Y) \le S_n$, 
then $\| P_Y b_n P_Y \| < 1/3$;
\item
$\lim_n S_n = \infty$;
\item
the propagation of $b = \sum_n b_n$ is finite.
\end{enumerate}
\end{lemma}

This lemma is essentially the same as \cite[Lemma 4.2]{RoeWillett}. The author gives its proof for the reader's convenience.

\begin{proof}
Suppose that $X$ does not have the operator norm localization property.
By Lemma \ref{Lemma quantifier}, there exists a positive number $R$ such that condition ($\alpha$) does not hold for $c = 1/3$. For such $R$, we have condition ($\beta$): for every positive number $S$,
there exists an operator $a \in E_R$ such that
$\|a P_Y\| < \| a \| /3$ whenever the diameter of $Y \subset X$ is at most $S$.

We prove the lemma by induction.
Let $S_1$ be a positive number.
Choose an operator $a$ satisfying condition ($\beta$)
for $S = S_1$.
The operator $a^* a$ 
is a member of $E_{2 R}$
and satisfies $\| P_Y a^* a P_Y \| < \| a^* a \| / 9$ for every subset $Y \subset X$
 whose diameter is at most $S_1$.
On the other hand, there exists a finite set $V(1)$ of $X$
such that 
$\| a^* a \| / 3 <  \| P_{V(1)} a^* a P_{V(1)}\|$
and that $V(1)$ is coarsely connected.
For any subset $Y \subset V(1)$ whose diameter is at most $S_1$, the inequality
\[\| P_Y a^* a P_Y \|  < \| a^* a \|/9  < \| P_{V(1)} a^* a P_{V(1)}\|/3\]
holds.
Define $b_1$ by $P_{V(1)} a^* a P_{V(1)}$.

Now we assume that there exist operators $b_1, b_2, \cdots, b_n$, finite subsets $V(1)$, $V(2)$, $\cdots$, $V(n)$ of $X$, and positive numbers $S_1, \cdots, S_n$ with the following conditions.
\begin{enumerate}[(i)]
\item\label{(1)onlpwwexpander}
every $V(j)$ is coarsely connected.
\item\label{(2)onlpwwexpander}
$V(1), V(2), \cdots, V(n)$ are disjoint.
\item\label{(3)onlpwwexpander}
$b_1, b_2, \cdots, b_n$ are positive elements of $E_{2 R}$
and satisfy $b_j = P_{V(j)} b_j P_{V(j)}$.
\item\label{(4)onlpwwexpander}
$\| P_Y b_j P_Y \| 
< \| b_j \| /3$, when
the diameter of $Y \subset V(j)$ is at most $S_j$,
\item\label{(5)onlpwwexpander}
$S_j + 1 \le S_{j + 1}$.
\end{enumerate}
Define $S_{n + 1}$ by
\[S_{n + 1} = \max_{\widetilde{X}}\ \mathrm{diam} \left( \bigcup \left\{ V(j) \ \left| \ V(j) \subset \widetilde{X}, 1 \le j \le n \right. \right\} \right) + 1,\]
where $\widetilde{X}$ is a coarse connected component of $X$.

By condition ($\beta$), there exists an operator $a \in E_R$ such that
$\|a P_Y\| < \| a \| / 3$ whenever the diameter of $Y \subset X$ 
is at most $S_{n + 1}$.
For every coarse connected component $\widetilde{X} \subset X$, the projection $P_{\widetilde {X}}$ commutes with $a$ and 
\[\mathrm{diam} \left( \bigcup \left\{ V(j) \ \left| \ 
V(j) \subset \widetilde{X}, 1 \le j \le n \right. \right\} \right) < S_{n + 1}.\]
Therefore, we have
$\| a P _{\cup_{j = 1}^n V(j)}\| < \|a\| / 3$.
There exists a finite subset $V(n + 1)$ of $X \setminus \cup_{j = 1}^n V(j)$ such that:
\begin{itemize}
\item
$\frac{\sqrt 3}{2} \left\| a P _{X \setminus \cup_{j = 1}^n V(j)} \right\|
\le \left\| a P _{V(n + 1)} \right\|$;
\item
$V(n + 1)$ is coarsely connected.
\end{itemize}
The inequality in the first item implies
\begin{eqnarray*}
\frac{1}{\sqrt 3} \|a\|
&=&
\frac{\sqrt 3}{2} \left( \|a\| - \frac{1}{3} \left\| a \right\| \right) 
<
\frac{\sqrt 3}{2} \left( \|a\| - \left\| a P _{\cup_{j = 1}^n V(j)} \right\| \right) \\
&\le& 
\frac{\sqrt 3}{2} \left\|a P _{X \setminus \cup_{j = 1}^n V(j)} \right\|
\le
\left\| a P _{V(n + 1)} \right\|.
\end{eqnarray*}
Define $b_{n + 1}$ by $P _{V(n + 1)} a^* a P _{V(n + 1)}$.
For every subset $Y \subset V(n + 1)$ whose diameter is at most $S_{n + 1}$, we have
\begin{eqnarray*}
\|P_Y b_{n + 1} P_Y\| = \|P_Y a^* a P_Y\| < \frac{\| a^* a \|}{9} 
< \frac{\left\| P _{V(n + 1)} a^* a P _{V(n + 1)} \right\|}{3} = \frac{\| b_{n + 1} \|}{3}.
\end{eqnarray*}
Now we obtain operators $b_1$, $\cdots$, $b_n$, $b_{n+1}$,
finite subsets $V(1)$, $\cdots$, $V(n)$, $V(n + 1)$,
and positive numbers
$S_1$, $\cdots$, $S_n$, $S_{n + 1}$ with the same conditions as (\ref{(1)onlpwwexpander}),
(\ref{(2)onlpwwexpander}),
(\ref{(3)onlpwwexpander}),
(\ref{(4)onlpwwexpander}),
and
(\ref{(5)onlpwwexpander}).

Iterate this procedure infinitely many times.
The sequence $(S_n)$ diverges and
the propagation of $b := \sum_{j = 1}^\infty b_j / \|b_j\|$ is at most $2R$.
\end{proof}

\subsection{Local reflexivity}

In this paper, we consider the following properties for C$^*$-algebras:
nuclearity,
exactness,
and local reflexivity.
It is not hard to see that nuclearity implies exactness. Conclusions by Kirchberg (\cite{KirchbergExactUHF} and \cite{KirchbergCAR}) show that
exactness implies local reflexivity.
They are all related to minimal tensor products 
between C$^*$-algebras. 
The following is the definition of local reflexivity.
\begin{definition}[Section 5 of Effros--Haagerup \cite{EffrosHaagerup}]
A C$^*$-algebra $B$ is said to be locally reflexive if for every finite-dimensional operator system $V \subset B^{**}$, 
there exists a net of contractive completely positive maps $\Phi_\iota \colon V \rightarrow B$ 
which converges to $\mathrm{id}_V$ in the point-weak-$*$ topology.
\end{definition}
We use the following proved in
\cite[Theorem 3.2, 5.1, and Proposition 5.3]{EffrosHaagerup}:
\begin{itemize}
\item
A C$^*$-subalgebra of a locally reflexive C$^*$-algebra is also locally reflexive.
\item
If $B$ is locally reflexive and $K$ is an ideal of $B$, 
then the exact sequence 
$0 \rightarrow K \rightarrow B \rightarrow B/K \rightarrow 0$ 
locally splits.
More precisely, for every finite-dimensional operator system 
$V \subset B / K$, 
there exists a unital completely positive map $\Phi \colon V \rightarrow B$ such that $\Phi(a) + K = a$, $a \in V$. 
\end{itemize}
If the algebra $B$ is locally reflexive and $K$ is an ideal of $B$,
then for every C$^*$-algebra $C$, the naturally defined sequence
$0 
\rightarrow K \otimes_\mathrm{min} C 
\rightarrow B \otimes_\mathrm{min} C 
\rightarrow B/K \otimes_\mathrm{min} C 
\rightarrow 0$
is exact (\cite[Theorem 3.2]{EffrosHaagerup}).

\section{Uniform Roe algebra of a sequence of finite connected graphs}
\label{section: sequence of finite graphs}

In this section, we concentrate on the case that an infinite metric space $V$ with bounded geometry is given by a sequence of finite connected graphs $(G_n$ $=$ $(V_n,$ ${\rm Edge}_n))_{n \in \N}$.
We denote by $V$ the disjoint union $\sqcup_n V_n$.
We study a sequence of vector states $\omega_n = \langle \cdot \xi_n, \xi_n \rangle$ on $\CstarU(V)$ given by unit vectors $\xi_n \in \ell_2(V_n)$
under the assumption that $\CstarU(V)$ is locally reflexive.
Our goal is Theorem \ref{theorem: approximation}.
Theorem \ref{theorem: approximation} means that the vector states can be approximated by some states which satisfy some localization property.

\subsection{Action of the $\Cstar$-minimal tensor product}

Let $(G_n = (V_n, {\rm Edge}_n))_{n \in \N}$ be a sequence of finite connected graphs.
The set $V_n$ is the vertex set of $G_n$, and ${\rm Edge}_n \subset V_n^2$ is a symmetric subset of $V_n$, which is called an edge set of $G_n$.
The graph metric $d_n$ on $V_n$ is defined. 
It is the maximal metric function on $V_n \times V_n$ satisfying that
$d_n(x, y) \le 1$ for $(x, y) \in \mathrm{Edge}_n$.

We also introduce a metric $d$ on $V = \sqcup_n V_n$ by the following:
for $x, y  \in V$,
\begin{itemize}
\item
if there exists $n \in \N$ such that $x, y \in V_n$, then $d(x, y) = d_n(x, y)$,
\item
if $x \in V_m$, $y \in V_n$, and if $m \neq n$, then $d(x, y) = \infty$.
\end{itemize}
We always assume that $V$ has bounded geometry and that $\CstarU(V)$ is locally reflexive.
We note that the algebra $\CstarU(V)$ is a subalgebra of $\prod_{n \in \N} \B(\ell_2(V_n))$, the algebra of norm bounded sequences of matrices.

The algebra $\CstarU(V)$ acts on the Hilbert space $\bigoplus_{n \in \N} \ell_2(V_n)$.
The opposite algebra $\CstarU(V)^\mathrm{op}\subset \prod_{n \in \N} \B(\ell_2(V_n))^\mathrm{op}$ also acts on $\bigoplus_{n \in \N} \ell_2(V_n)$ by its transpose:
\[(a_n)_n^\mathrm{op}(\xi_n)_n = \left( a_n^\mathrm{T} \xi_n \right)_n, 
\quad (a_n)_n^\mathrm{op} \in \CstarU(V)^\mathrm{op}, 
(\xi_n)_n \in \bigoplus_{n \in \N} \ell_2(V_n).\]
The $\Cstar$-minimal tensor product 
$\CstarU(V) 
\otimes_\mathrm{min} 
\CstarU(V)^\mathrm{op}$
acts on the Hilbert space 
$\bigoplus_{n \in \N} \ell_2(V_n) \otimes \bigoplus_{n \in \N} \ell_2(V_n)$.
The closed subspace $\bigoplus_{n \in \N} \ell_2 \left( V_n^2 \right)$ of $\bigoplus_{n \in \N} \ell_2(V_n) \otimes \bigoplus_{n \in \N} \ell_2(V_n)$ is invariant under the action of the algebra $\CstarU(V) 
\otimes_\mathrm{min} 
\CstarU(V)^\mathrm{op}$.

\subsection{A state $\phi_\infty$ of $\CstarU(V) 
\otimes_\mathrm{min}
\CstarU(V)^\mathrm{op}$}
\label{subsection: diagonal state at infinity}

Let $(\xi_n)_n$ be a sequence of unit vectors in the Hilbert space $\bigoplus_{n \in \N} \ell_2(V_n)$.
Until Theorem \ref{theorem: approximation}, we always assume the following:
\begin{itemize}
\item
$\xi_n$ is a positive element of $\ell_2(V_n)$ and its support is $V_n$,
\item
there exists a positive number $\Lambda$ such that for every $n \in \N$ and for every $(x, y) \in {\rm Edge}_n$,
$\xi_n(x)^2 \le \Lambda \xi_n(y)^2$.
\end{itemize}
In the proof of Theorem \ref{theorem: approximation}, we prove that the general case reduces to the case that the unit vectors satisfy these conditions.

Define a sequence of unit vectors $(\Xi_n)_n$ in $\bigoplus_{n \in \N} \ell_2 \left( V_n^2 \right)$ by
\[\Xi_n = \sum_{x \in V_n} \xi_n(x) \delta_{(x, x)} \in \ell_2 \left( V_n^2 \right).\]
These vectors $\Xi_n$ define states $\phi_n(\ \cdot \ ) = \langle \ \cdot \ \Xi_n, \Xi_n \rangle$ on the $\Cstar$-algebra $\CstarU(V) 
\otimes_\mathrm{min} 
\CstarU(V)^\mathrm{op}$.
Although our subject is the sequence of vector states $\omega_n = \langle \cdot \xi_n, \xi_n \rangle$,
we concentrate on 
that of vector states $\phi_n = \langle \ \cdot \ \Xi_n, \Xi_n \rangle$ for a while.
Choose and fix an accumulation point $\phi_\infty$ of the sequence $(\phi_n)_n$ in the dual space $\left(\CstarU(V) 
\otimes_\mathrm{min} 
\CstarU(V)^\mathrm{op} \right)^* $.
Denote by $\left( \pi_\infty, \Hil_\infty, \Xi_\infty \right)$ the GNS-triple of the state $\phi_\infty$ on the algebra $\CstarU(V) 
\otimes_\mathrm{min} 
\CstarU(V)^\mathrm{op}$.
For $a \in \CstarU(V)$ and $a^\mathrm{op} \in \CstarU(V)^\mathrm{op}$, we simply write
\begin{eqnarray*}
\pi_\infty(a \otimes 1) &=& \pi_\infty(a),\\
\pi_\infty(1 \otimes a^\mathrm{op}) &=& \pi_\infty(a)^\mathrm{op}.
\end{eqnarray*}
The image $\pi_\infty(\CstarU(V))$ commutes with $\pi_\infty(\CstarU(V))^\mathrm{op}$.

\begin{lemma}
\label{lemma: left and right}
For every $a \in \cup_R E_R$,
there exists $b \in \cup_R E_R$
such that $\pi_\infty(a) \Xi_\infty = \pi_\infty(b)^\mathrm{op} \Xi_\infty$.
For every $b \in \cup_R E_R$,
there exists $a \in \cup_R E_R$
such that $\pi_\infty(a) \Xi_\infty = \pi_\infty(b)^\mathrm{op} \Xi_\infty$.
\end{lemma}

\begin{proof}
We prove the first half.
Let $([a_{n, x, y}]_{x, y \in V_n})_n$ be the matrix expression of a non-zero operator $a \in \CstarU(V)$ with finite propagation. Define $R$ by $R = \sup \{d(x, y) \ |\ n \in \N, x, y \in V_n, a_{n, x, y} \neq 0\}$.
Define $b \in \CstarU(V)$ by
\[b = \left( \left[ a_{n, x, y} \frac{\xi_n(y)}{\xi_n(x)} \right]_{x, y \in V_n} \right)_n\]
The vector $(a \otimes 1) \Xi_n$ is equal to the following vector:
\begin{eqnarray*}
\sum_{z \in V_n} \left( [a_{n, x, y}]_{x, y \in V_n} \xi_n(z) \delta_z \right) \otimes \delta_z 
&=& 
\sum_{z \in V_n} \left( \sum_{w \in V_n} a_{n, w, z} \xi_n(z) \delta_w \right) \otimes \delta_z\\
&=& 
\sum_{w \in V_n} \delta_w \otimes 
\left(\sum_{z \in V_n} a_{n, w, z} \frac{\xi_n(z)}{\xi_n(w)} \xi_n(w) \delta_z \right)\\
&=& 
\sum_{w \in V_n} \delta_w \otimes 
\left( \left[a_{n, x, y} \frac{\xi_n(y)}{\xi_n(x)} \right]_{x, y \in V_n}^\mathrm{T} \xi_n(w) \delta_w \right)\\
&=&
(1 \otimes b^\op) \Xi_n.
\end{eqnarray*}
It follows that $\phi_n( |(a \otimes 1) - (1 \otimes b^\op)|^2) = 0$.
Therefore we have $\phi_\infty( |(a \otimes 1) - (1 \otimes b^\op)|^2) = 0$ and 
$\pi_\infty(a \otimes 1) \Xi_\infty - \pi_\infty(1 \otimes b^\op) \Xi_\infty = 0$.
\end{proof}

\begin{lemma}
\label{lemma: dense subspaces}
The vector $\Xi_\infty$ is a cyclic vector of $\pi_\infty(\CstarU(V))$.
The vector $\Xi_\infty$ is a cyclic vector of $\pi_\infty(\CstarU(V))^\op$.
\end{lemma}

\begin{proof}
We prove the second half.
The subspace 
$\mathrm{span}\{\pi_\infty(b)^\op \pi_\infty(a) \Xi_\infty \ | \ a, b \in \cup_R E_R\}$
of $\Hil_\infty$ is dense. 
By Lemma \ref{lemma: left and right}, for every $a \in  \cup_R E_R$,
there exists $b_1 \in \cup_R E_R$ satisfying that $\pi_\infty(a) \Xi_\infty = \pi_\infty(b_1)^\op \Xi_\infty$.
Therefore, the subspace
$\{\pi_\infty(b)^\op \Xi_\infty \ | \ b \in \cup_R E_R\}$ is equal to the above dense subspace.
\end{proof}

\subsection{A representation of $\ell_\infty \left( \sqcup_{n} V_n^2 \right)$ on $\Hil_\infty$}

Let $\zeta$ be a bounded function on $\sqcup_{n} V_n^2$.
For an arbitrary bounded operator $a = ([a_{n, x, y}]_{x, y \in V_n})_n$ on $\ell_2(V)$ with finite propagation, 
since $|\zeta(x, y) a_{n, x, y}| \le \|\zeta\|_\infty \|a\|$,
the operator $([\zeta(x, y) a_{n, x, y}]_{x, y \in V_n})_n$ is also a bounded operator with finite propagation, by Lemma \ref{lemma: operator given by bounded controlled function}. We denote by $M[\zeta] (a)$ the operator.

\begin{lemma}
\label{lemma: varpi_infty}
There exists a unique bounded operator $\varpi_\infty(\zeta)$ on $\Hil_\infty$ satisfying that for every $a \in \cup_R E_R$, $\varpi_\infty(\zeta) \pi_\infty(a) \Xi_\infty = \pi_\infty(M[\zeta](a)) \Xi_\infty$.
\end{lemma}

\begin{proof}
For every $a = ([a_{n, x, y}]_{x, y \in V_n})_n \in \cup_R E_R$,
there exists an increasing sequence $(n(k))$ of natural numbers such that
\begin{eqnarray*}
\phi_\infty( a^* a )
&=&
\lim_{k \to \infty} \phi_{n(k)}( a^* a ),
\\
\phi_\infty( M[\zeta](a)^* M[\zeta](a))
&=&
\lim_{k \to \infty} \phi_{n(k)}( M[\zeta](a)^* M[\zeta](a)).
\end{eqnarray*}
For every $k$, we have
\begin{eqnarray*}
\phi_{n(k)}( a^* a )
&=& \sum_{y \in V_n} \xi_{n(k)}(y)^2
\sum_{x \in V_n} |a_{n, x, y}|^2,\\
\phi_{n(k)}( M[\zeta](a)^* M[\zeta](a))
&=& \sum_{y \in V_n} \xi_{n(k)}(y)^2 
\sum_{x \in V_n} |\zeta(x, y)|^2 |a_{n, x, y}|^2.
\end{eqnarray*}
It follows that $\phi_{n(k)}( M[\zeta](a)^* M[\zeta](a))
\le \|\zeta\|_\infty^2 \phi_{n(k)}( a^* a )$.
Thus we obtain the inequalities 
$\phi_\infty( M[\zeta](a)^* M[\zeta](a))
\le \|\zeta\|_\infty^2 \phi_\infty( a^* a )$ and
\begin{eqnarray*}
\|\pi_\infty(M[\zeta](a)) \Xi_\infty\|^2
\le 
\| \zeta \|_\infty^2
\|\pi_\infty(a) \Xi_\infty\|^2.
\end{eqnarray*}
We conclude that the mapping $\pi_\infty(a) \Xi_\infty \mapsto \pi_\infty(M[\zeta](a)) \Xi_\infty$ is well defined and bounded.
By Lemma \ref{lemma: dense subspaces},
the mapping uniquely extends to a bounded operator on the Hilbert space $\Hil_\infty$.
\end{proof}
It is also easy to show that
$\varpi_\infty \colon \ell_\infty \left( \sqcup_{n} V_n^2 \right) \to \B(\Hil_\infty)$ is a $*$-homomorphism.

Note that
there exists a natural unital embedding $\iota$ of $\ell_\infty(V)$ into $\ell_\infty \left( \sqcup_{n} V_n^2 \right)$ defined by
\[ [\iota(\zeta)](x, y) = \zeta(x), \quad \zeta \in \ell_\infty(V), (x, y) \in \sqcup_{n} V_n^2.\]
It straightforward to show that the representation $\pi_\infty |_{\ell_\infty(V)}$ of the diagonal subalgebra $\ell_\infty(V)$ of $\CstarU(V)$ and the representation $\varpi_\infty \circ \iota$ are identical.

\subsection{Modular operator with respect to $(\pi_\infty(\CstarU(V))'', \Xi_\infty)$}

Recall that $E_R$ is the space of operators on $\ell_2(V)$ whose propagations are at most $R$.
Denote by $\Hil_R$ the closure of $\pi_\infty(E_R) \Xi_\infty \subset \Hil_\infty$.
Note that $\cup_R \Hil_R$ is dense in $\Hil_\infty$, by Lemma \ref{lemma: dense subspaces}.
Define a function $H \colon \sqcup_n V_n^2 \to \R$ by
\[H(x, y) = \xi_n(x)^2 / \xi_n(y)^2, \quad n \in \N, x, y \in V_n.\]
Note that the function $H$ is bounded on every controlled subset of $\sqcup_n V_n^2$.
Therefore, the operator $M[H]$ defines a bounded self-adjoint operator on $\Hil_R$ for every $R$.
Such an operator on $\cup_R \Hil_R$ uniquely extends to a self-adjoint operator on $\Hil_\infty$, which is not necessarily bounded.
Denote by $\Delta$ the extension.

We also define an anti-linear isometry $J$ on $\Hil_\infty$ by the following.

\begin{lemma}
There exists a unique anti-linear isometry $J$ such that for every
$a = ([a_{n, x, y}]_{x, y \in V_n})_n \in \cup_R E_R$,
\[J \pi_\infty(([a_{n, x, y}]_{x, y \in V_n})_n) \Xi_\infty = 
\pi_\infty \left( \left( \left[ H(x, y)^{1/2}\ \overline{a_{n, y, x}} \right]_{x, y \in V_n} \right)_n \right) \Xi_\infty.\]
\end{lemma}

\begin{proof}
Because the sequence of matrices $\left( \left[ H(x, y)^{1/2} \overline{a_{n, y, x}} \right]_{x, y \in V_n} \right)_n$ has finite propagation and the set of entries is bounded,
it gives a bounded operator by Lemma \ref{lemma: operator given by bounded controlled function}.
It is straightforward to show that for every $n$,
\[
\left\|
[a_{n, x, y}]_{x, y \in V_n} \Xi_n
\right\|^2
=
\left\|
\left[ H(x, y)^{1/2}\ \overline{a_{n, y, x}} \right]_{x, y \in V_n} \Xi_n
\right\|^2,
\]
\begin{eqnarray*}
&& \phi_n( ([a_{n, x, y}]_{x, y \in V_n})_n^* ([a_{n, x, y}]_{x, y \in V_n})_n )\\
&=&
\phi_n \left( \left( \left[ H(x, y)^{1/2}\ \overline{a_{n, y, x}} \right]_{x, y \in V_n} \right)_n^*
\left( \left[ H(x, y)^{1/2}\ \overline{a_{n, y, x}} \right]_{x, y \in V_n} \right)_n \right).
\end{eqnarray*}
The latter equality implies
\begin{eqnarray*}
&& \phi_\infty( ([a_{n, x, y}]_{x, y \in V_n})_n^* ([a_{n, x, y}]_{x, y \in V_n})_n )\\
&=&
\phi_\infty \left( \left( \left[ H(x, y)^{1/2}\ \overline{a_{n, y, x}} \right]_{x, y \in V_n} \right)_n^*
\left( \left[ H(x, y)^{1/2}\ \overline{a_{n, y, x}} \right]_{x, y \in V_n} \right)_n \right),
\end{eqnarray*}
\begin{eqnarray*}
\left\|
\pi_\infty(([a_{n, x, y}]_{x, y \in V_n})_n) \Xi_\infty
\right\|^2
&=&
\left\|
\pi_\infty \left( \left( \left[ H(x, y)^{1/2}\ \overline{a_{n, y, x}} \right]_{x, y \in V_n} \right)_n \right) \Xi_\infty 
\right\|^2.
\end{eqnarray*}
Because $\pi_\infty(\cup_R E_R) \Xi_\infty$ is dense in $\Hil_\infty$ (Lemma \ref{lemma: dense subspaces}),
there exists a unique isometry $J$ on $\Hil_\infty$ which extends
the mapping
\begin{eqnarray*}
\pi_\infty(([a_{n, x, y}]_{x, y \in V_n})_n) \Xi_\infty
&\mapsto&
\pi_\infty \left( \left( \left[ H(x, y)^{1/2}\ \overline{a_{n, y, x}} \right]_{x, y \in V_n} \right)_n \right) \Xi_\infty.
\end{eqnarray*}
It is also clear that $J$ is anti-linear.
\end{proof}

\begin{lemma}
\label{lemma: relation between J and Delta}
$J^2 = 1$, $J \Delta^{1/2} = \Delta^{-1/2} J$.
\end{lemma}

\begin{proof}
For every $([a_{n, x, y}]_{x, y \in V_n})_n \in \cup_R E_R$, we obtain
\begin{eqnarray*}
&&
J J \pi_\infty(([a_{n, x, y}]_{x, y \in V_n})_n) \Xi_\infty\\
&=&
J \pi_\infty \left( \left( \left[ H(x, y)^{1/2}\ \overline{a_{n, y, x}} \right]_{x, y \in V_n} \right)_n \right) \Xi_\infty\\
&=&
\pi_\infty \left( \left( \left[ H(x, y)^{1/2}\ \overline{H(y, x)^{1/2}\ \overline{a_{n, x, y}}} \right]_{x, y \in V_n} \right)_n \right) \Xi_\infty\\
&=&
\pi_\infty(([a_{n, x, y}]_{x, y \in V_n})_n) \Xi_\infty.
\end{eqnarray*}
Therefore we have $JJ = 1$. We also obtain
\begin{eqnarray*}
&&
J \Delta^{1/2}
\pi_\infty(([a_{n, x, y}]_{x, y \in V_n})_n) \Xi_\infty\\
&=&
J \pi_\infty \left( \left( \left[ H(x, y)^{1/2}\ a_{n, x, y} \right]_{x, y \in V_n} \right)_n \right) \Xi_\infty \\
&=&
\pi_\infty \left( \left( \left[ H(x, y)^{1/2}\ \cdot \overline{ H(y, x)^{1/2}\ a_{n, y, x}} \right]_{x, y \in V_n} \right)_n \right) \Xi_\infty \\
&=&
\pi_\infty \left( \left( \left[ \overline{a_{n, y, x}} \right]_{x, y \in V_n} \right)_n \right) \Xi_\infty.
\end{eqnarray*}
and
\begin{eqnarray*}
\Delta^{- 1/2} J 
\pi_\infty(([a_{n, x, y}]_{x, y \in V_n})_n) \Xi_\infty
&=&
\Delta^{- 1/2} \pi_\infty \left( \left( \left[ H(x, y)^{1/2}\ \overline{a_{n, y, x}} \right]_{x, y \in V_n} \right)_n \right) \Xi_\infty \\
&=&
\pi_\infty \left( \left( \left[ \overline{a_{n, y, x}} \right]_{x, y \in V_n} \right)_n \right) \Xi_\infty.
\end{eqnarray*}
Therefore we have $J \Delta^{1/2} = \Delta^{-1/2} J$ on every $\Hil_R$.
It implies the lemma.
\end{proof}

\begin{lemma}
\label{lemma: Delta and the opposite representation}
For every $b \in \cup_R E_R$, 
$\pi_\infty(b)^\op \Xi_\infty = \Delta^{1/2} \pi_\infty(b) \Xi_\infty$.
\end{lemma}

\begin{proof}
This has been already shown in the proof of Lemma \ref{lemma: left and right}.
\end{proof}

\begin{lemma}
\label{lemma: J, Delta, and adjoint}
For every $a, b \in \cup_R E_R$, 
the following equations hold:
\[\pi_\infty(a^*) \Xi_\infty = J \Delta^{1/2} \pi_\infty(a) \Xi_\infty, 
\quad
\pi_\infty(b^*)^\op \Xi_\infty 
= J \Delta^{-1/2} \pi_\infty(b)^\op \Xi_\infty.\]
\end{lemma}

\begin{proof}
For $a = ([a_{n, x, y}]_{x, y \in V_n})_n$, we have already obtained the equation
\begin{eqnarray*}
J \Delta^{1/2}
\pi_\infty(([a_{n, x, y}]_{x, y \in V_n})_n) \Xi_\infty
=
\pi_\infty \left( \left( \left[ \overline{a_{n, y, x}} \right]_{x, y \in V_n} \right)_n \right) \Xi_\infty,
\end{eqnarray*}
in the proof of Lemma \ref{lemma: relation between J and Delta}.
This is the first claim.
By Lemma \ref{lemma: Delta and the opposite representation}, 
and by the first claim,
we obtain
\begin{eqnarray*}
\pi_\infty(b^*)^\op \Xi_\infty 
= \Delta^{1/2} \pi_\infty(b^*) \Xi_\infty
= \Delta^{1/2} J \Delta^{1/2} \pi_\infty(b) \Xi_\infty.
\end{eqnarray*}
By Lemma \ref{lemma: Delta and the opposite representation} and 
\ref{lemma: relation between J and Delta},
we obtain
\begin{eqnarray*}
\Delta^{1/2} J \Delta^{1/2} \pi_\infty(b) \Xi_\infty
= \Delta^{1/2} J \pi_\infty(b)^\op \Xi_\infty
= J \Delta^{- 1/2} \pi_\infty(b)^\op \Xi_\infty.
\end{eqnarray*}
Thus we obtain the second claim.
\end{proof}

\begin{lemma}
\label{lemma: J and the opposite representation}
For every $a \in \cup_R E_R$, 
$J \pi_\infty(a) J = \pi_\infty(a^*)^\op$.
\end{lemma}

\begin{proof}
For every $a, b \in \cup_R E_R$, 
by Lemma \ref{lemma: Delta and the opposite representation},
we have
\begin{eqnarray*}
J \pi_\infty(a) J \pi_\infty(b)^\op \Xi_\infty
=
J \pi_\infty(a) J \Delta^{1/2} \pi_\infty(b) \Xi_\infty.
\end{eqnarray*}
By Lemma
\ref{lemma: J, Delta, and adjoint}, we have
\begin{eqnarray*}
J \pi_\infty(a) J \Delta^{1/2} \pi_\infty(b) \Xi_\infty
=
J \pi_\infty(a) \pi_\infty(b^*) \Xi_\infty
=
J \pi_\infty(a b^*) \Xi_\infty.
\end{eqnarray*}
Again by
Lemma \ref{lemma: Delta and the opposite representation},
we have
\begin{eqnarray*}
J \pi_\infty(a b^*) \Xi_\infty
=
J \Delta^{-1/2} \pi_\infty(a b^*)^\op \Xi_\infty.
\end{eqnarray*}
By Lemma
\ref{lemma: J, Delta, and adjoint}, we have
\begin{eqnarray*}
J \Delta^{-1/2} \pi_\infty(a b^*)^\op \Xi_\infty
=
\pi_\infty(b a^*)^\op \Xi_\infty
=
\pi_\infty(a^*)^\op \pi_\infty(b)^\op \Xi_\infty.
\end{eqnarray*}
Since $\pi_\infty(\cup_R E_R)^\op \Xi_\infty$ is dense in $\Hil_\infty$ (Lemma \ref{lemma: dense subspaces}),
we obtain
$J \pi_\infty(a) J = \pi_\infty(a^*)^\op$.
\end{proof}

Denote by $S$ the closure of the mapping
\begin{eqnarray*}
\pi_\infty(\CstarU(V))'' \Xi_\infty &\to& \pi_\infty(\CstarU(V))'' \Xi_\infty\\
a \Xi_\infty &\mapsto& a^* \Xi_\infty.
\end{eqnarray*}
Denote by $F$ the closure of the mapping
\begin{eqnarray*}
\pi_\infty(\CstarU(V))' \Xi_\infty &\to& \pi_\infty(\CstarU(V))' \Xi_\infty\\
b \Xi_\infty &\mapsto& b^* \Xi_\infty.
\end{eqnarray*}
It is straightforward to show that $F \subset S^*$ and $S \subset F^*$.

\begin{lemma}
\label{lemma: polar decompositions}
$S = J \Delta^{1/2}, F = J \Delta^{-1/2}$.
\end{lemma}

\begin{proof}
By Lemma \ref{lemma: J, Delta, and adjoint},
the operators $S$ and $J \Delta^{1/2}$ are identical on the subspace $\cup_R \pi_\infty(E_R) \Xi_\infty$.
The subspace $\cup_R \pi_\infty(E_R) \Xi_\infty$ is a core of $\Delta^{1/2}$.
It follows that $J \Delta^{1/2} \subset S$.
Recall that $\pi_\infty(\cup_R E_R)^{\mathrm{op}} \subset \pi_\infty(\CstarU(V))^{\prime}$.
By Lemma \ref{lemma: J, Delta, and adjoint},
the operators $F$ and $J \Delta^{- 1/2}$ are identical on the subspace $\cup_R \pi_\infty(E_R)^\op \Xi_\infty$.
Since $\pi_\infty(\cup_R E_R)^\op \Xi_\infty = \pi_\infty(\cup_R E_R) \Xi_\infty$ (Lemma \ref{lemma: left and right}),
$\pi_\infty(\cup_R E_R)^\op \Xi_\infty$ is a core of $\Delta^{- 1/2}$.
It follows that $J \Delta^{- 1/2} \subset F$.
Taking the adjoints of $J \Delta^{1/2} \subset S$, we obtain $S^* \subset (J \Delta^{1/2})^*$.
Combining with Lemma \ref{lemma: relation between J and Delta}, we have
\[J \Delta^{- 1/2} \subset F \subset S^* \subset (J \Delta^{1/2})^* = \Delta^{1/2} J = J \Delta^{-1/2}.\]
Thus we obtain $F = J \Delta^{-1/2}$.

Taking the adjoints of $J \Delta^{- 1/2} \subset F$, we obtain $F^* \subset (J \Delta^{- 1/2})^*$.
Combining with Lemma \ref{lemma: relation between J and Delta}, we have
\[J \Delta^{1/2} \subset S \subset F^* \subset (J \Delta^{- 1/2})^* = \Delta^{- 1/2} J = J \Delta^{1/2}.\]
Thus we obtain $S = J \Delta^{1/2}$.
\end{proof}

\begin{proposition}\label{proposition: commutant}
${\pi_\infty(\CstarU(V))^\op}'' = \pi_\infty(\CstarU(V))'$.
\end{proposition}

\begin{proof}
By Lemma \ref{lemma: polar decompositions},
$J$ is the modular conjugation of $\pi_\infty(\CstarU(V))''$ with respect to the state given by $\Xi_\infty$.
By the fundamental theorem by Tomita \cite[Theorem 4.1]{TomitaTakesaki} for the left Hilbert algebra $\pi(\CstarU(V))'' \Xi_\infty$,
the commutant $\pi_\infty(\CstarU(V))'$ is identical to $J \pi_\infty(\CstarU(V))'' J$.
It follows that $\pi_\infty(\CstarU(V))'$ is generated by $J \pi_\infty(\cup_R E_R) J$.
By Lemma \ref{lemma: J and the opposite representation},
$J \pi_\infty(\cup_R E_R) J = \pi_\infty(\cup_R E_R)^\op$.
We conclude that $\pi_\infty(\CstarU(V))^\op$ generates the commutant $\pi_\infty(\CstarU(V))'$.
\end{proof}

\subsection{Existence of an equivariant state $\psi$ on $\B(\Hil_\infty)$}
The strategy used in the first half of this subsection is called ``The Trick'' in the book \cite[Section 3.6]{OzawaBook}.
Denote by $K$ the kernel of the representation $\pi_\infty \colon \CstarU(V) \to \B(\Hil_\infty)$.
Suppose that $\CstarU(V)$ is locally reflexive.
Then the sequence
\[0 
\to K  \otimes_\mathrm{min} \CstarU(V)^\mathrm{op} 
\to \CstarU(V)  \otimes_\mathrm{min} \CstarU(V)^\mathrm{op} 
\to 
\pi_\infty(\CstarU(V))  \otimes_\mathrm{min} \CstarU(V)^\mathrm{op} 
\to 0\]
is exact.
Since 
the representation 
\[ \pi_\infty \colon \CstarU(V) 
\otimes_\mathrm{min} 
\CstarU(V)^\mathrm{op} \to \B(\Hil_\infty)\]
is zero on $K \otimes_\mathrm{min} 
\CstarU(V)^\mathrm{op}$, it
naturally induces a representation
\[\pi_\infty(\CstarU(V))  \otimes_\mathrm{min} \CstarU(V)^\mathrm{op} 
\to \B(\Hil_\infty).\]

By the extension theorem by Arveson,
the representation
\[\pi_\infty(\CstarU(V)) \otimes_\mathrm{min} \CstarU(V)^\mathrm{op}
\to \B(\Hil_\infty)\]
can be extended to a unital completely positive map
\[\Phi \colon \B(\Hil_\infty) \otimes_\mathrm{min} \CstarU(V)^\mathrm{op}
\to \B(\Hil_\infty).\]
Since the subalgebra $\C \otimes_\mathrm{min} \CstarU(V)^\mathrm{op}$ is in the multiplicative domain of $\Phi$, 
$\Phi$ is a $( \C \otimes_\mathrm{min} \CstarU(V)^\mathrm{op} )$-module map.
It follows that the image $\Phi(\B(\Hil_\infty))$ of the first factor is included in the von Neumann algebra $\pi_\infty(\CstarU(V))^{\op \prime}$. By Proposition \ref{proposition: commutant}, the von Neumann algebra is identical to $\pi_\infty(\CstarU(V))^{\prime \prime}$.
Recycling the notation, we also denote by $\Phi$ the restriction
\[\Phi \colon \B(\Hil_\infty) \to \pi_\infty(\CstarU(V))''.\]
Note that $\Phi$ is the identity map on $\pi_\infty(\CstarU(V))$ and therefore $\Phi$ is a $\pi_\infty(\CstarU(V))$-module map.
Define a state $\psi$ on $\B(\Hil_\infty)$ by $\langle \Phi( \cdot ) \Xi_\infty, \Xi_\infty \rangle$.

We will show that
the state $\psi \circ \varpi_\infty$ on $\ell_\infty \left( \sqcup_{n} V_n^2 \right)$
is something like an invariant mean on an amenable group in Lemma \ref{lemma: equivariant state}.
To state the claim, we need to define several notations.
Take an arbitrary sequence $(\sigma_n)_n$ of partially defined injections
\[\sigma_n \colon \mathrm{dom} (\sigma_n) \to \mathrm{image} (\sigma_n),\]
from a subset of $V_n$ to a subset of $V_n$.
By taking the union, we obtain a partially defined injection
\[\sigma \colon \sqcup_n \mathrm{dom}(\sigma_n) \to \sqcup_n \mathrm{image}(\sigma_n),\]
defined on the subset $\sqcup_n \mathrm{dom}(\sigma_n)$ of $V = \sqcup_n V_n$.
We always consider the case that there exists a positive constant $R$ independent of $n$ such that 
for arbitrary $x, y \in V_n$, if $\sigma_n(y) = x$, then $d_n(x, y) \le R$.
We call $\sigma$ a {\it controlled} partially defined injection on $V$.
The map $\sigma$ induces a partial isometry $v_\sigma$ acting on $\ell_2(V)$ defined by $v_\sigma(\delta_x) = \delta_{\sigma(x)}, x \in \mathrm{dom}(\sigma)$.
Since $\sigma$ is controlled, 
$v_\sigma$ has finite propagation, and therefore is an element of $\CstarU(V)$.
Denote by $\sigma \times \mathrm{id}$ the partially defined injection
\begin{eqnarray*}
\sigma \times \mathrm{id} 
\colon 
\sqcup_n (\mathrm{dom}(\sigma_n) \times V_n)
&\to&
\sqcup_n (\mathrm{image}(\sigma_n) \times V_n),\\
(x, y)
&\mapsto&
(\sigma(x), y),
\end{eqnarray*}
on $\sqcup_n V_n^2$.
Note that the action of $v_\sigma$ on the Hilbert space $\ell_2 (\sqcup_n V_n^2)$ is given by the partially defined injection $\sigma \times \mathrm{id}$.
The partially defined injection $\sigma \times \mathrm{id}$ induces the $*$-isomorphisms
\begin{eqnarray*}
(\sigma \times \mathrm{id})_* 
&\colon& 
\ell_\infty( \sqcup_n (\mathrm{dom}(\sigma_n) \times V_n))
\to
\ell_\infty( \sqcup_n (\mathrm{image}(\sigma_n) \times V_n)),\\
(\sigma \times \mathrm{id})^* 
&\colon& 
\ell_\infty( \sqcup_n (\mathrm{image}(\sigma_n) \times V_n))
\to
\ell_\infty( \sqcup_n (\mathrm{dom}(\sigma_n) \times V_n)).
\end{eqnarray*}

\begin{lemma}\label{lemma: normalizer}
For $\zeta \in \ell_\infty( \sqcup_n (\mathrm{dom}(\sigma_n) \times V_n))$, we have
\[\varpi_\infty((\sigma \times \mathrm{id})_* (\zeta))
=
\pi_\infty(v_\sigma) \varpi_\infty(\zeta) \pi_\infty(v_\sigma^*). 
\]
\end{lemma}

\begin{proof}
For every $a = ([a_{n, x, y}]_{x, y \in V_n})_n \in \cup_R E_R$, 
the matrix coefficients of 
the operator $v_\sigma^* ([a_{n, x, y}]_{x, y \in V_n})_n$ at $(x, y) \in V_n^2$ is given by
\begin{eqnarray*}
\begin{cases}
a_{n, \sigma(x), y}, \quad x \in \mathrm{dom}(\sigma_n),\\
0, \quad x \notin \mathrm{dom}(\sigma_n).
\end{cases}
\end{eqnarray*}
For simplicity of the notation, in the case that $x \notin \mathrm{dom}(\sigma_n)$,
we define $a_{n, \sigma(x), y}$ by $0$.
For $\zeta \in \ell_\infty( \sqcup_n (\mathrm{dom}(\sigma_n) \times V_n))$, 
we have
\begin{eqnarray*}
&&
\pi_\infty(v_\sigma) \varpi_\infty(\zeta) \pi_\infty(v_\sigma^*) \pi_\infty(a) \Xi_\infty \\
&=& \pi_\infty(v_\sigma) \varpi_\infty(\zeta) \pi_\infty(([a_{n, \sigma(x), y}]_{x, y \in V_n})_n) \Xi_\infty \\
&=& \pi_\infty(v_\sigma) \pi_\infty(([\zeta(x, y) a_{n, \sigma(x), y}]_{x, y \in V_n})_n) \Xi_\infty \\
&=& \pi_\infty(([\zeta(\sigma^{-1}(x), y) a_{n, x, y}]_{x, y \in V_n})_n) \Xi_\infty \\
&=& \varpi_\infty((\sigma \times \mathrm{id})_* (\zeta)) \pi_\infty(a) \Xi_\infty
\end{eqnarray*}
\end{proof}

Define a positive function $h_\sigma \in \ell_\infty(V)$ on $\sqcup_n \mathrm{dom}(\sigma_n) \subset V$ by
\[h_\sigma(x) = \dfrac{ \xi_n(\sigma(x)) ^2 }{ \xi_n(x) ^2} = H(\sigma(x), x), \quad x \in \mathrm{dom}(\sigma_n).\]
By the assumption in Subsection
\ref{subsection: diagonal state at infinity},
the function $h_\sigma$ is bounded.

\begin{lemma}\label{lemma: equivariant state}
The state $\psi$ on $\B(\Hil_\infty)$ satisfies the following.
\begin{enumerate}
\item
For every $a \in \CstarU (V)$, $\psi( \pi_\infty (a) ) = \phi_\infty (a)$.
\item
The state $\psi$ is {\rm equivariant} with respect to the Radon--Nikodym derivative $H$. More precisely,
for every controlled partially defined injection $\sigma$ on $V$,
for every $c \in \B(\Hil_\infty)$, we have
\[\psi(\pi_\infty(v_\sigma) c \pi_\infty(v_\sigma^*)) = \psi( \pi_\infty(h_\sigma) c).\]
\end{enumerate}
\end{lemma}
It turns out that the state $\psi = \left\langle \Phi(\cdot) \Xi_\infty, \Xi_\infty \right\rangle$ is something similar to the hypertrace in \cite[Section V]{ConnesInjectiveFactors}.

\begin{proof}
Since $\Phi$ is the identity map on $\pi_\infty(\CstarU(V))$,
for every $a \in \CstarU (V)$, we have
\[\psi( \pi_\infty (a) ) = \langle \Phi(\pi_\infty (a)) \Xi_\infty, \Xi_\infty \rangle
= \langle \pi_\infty (a) \Xi_\infty, \Xi_\infty \rangle = \phi_\infty(a).\]
Because $\Phi$ is a $\pi_\infty(\CstarU(V))$-module map, for $c \in \B(\Hil_\infty)$, we have
\begin{eqnarray*}
\psi( \pi_\infty(v_\sigma) c \pi_\infty(v_\sigma^*) )
&=&
\langle \Phi( \pi_\infty(v_\sigma) c \pi_\infty(v_\sigma^*) ) \Xi_\infty,  
\Xi_\infty \rangle \\
&=&
\langle \Phi( c ) \pi_\infty(v_\sigma^*) \Xi_\infty,  
\pi_\infty(v_\sigma^*) \Xi_\infty \rangle.
\end{eqnarray*}
By Lemma \ref{lemma: J, Delta, and adjoint},
we have
\begin{eqnarray*}
\psi( \pi_\infty(v_\sigma) c \pi_\infty(v_\sigma^*) )
&=&
 \langle \Phi(c) J \Delta^{1/2} \pi_\infty(v_\sigma) \Xi_\infty,  
J \Delta^{1/2} \pi_\infty(v_\sigma) \Xi_\infty \rangle\\
&=&
 \langle \Phi(c) J \pi_\infty(M[H^{1/2}](v_\sigma)) \Xi_\infty,  
J \pi_\infty(M[H^{1/2}](v_\sigma)) \Xi_\infty \rangle.
\end{eqnarray*}
Since $J$ is antilinear involutive isometry, we have
\begin{eqnarray*}
\psi( \pi_\infty(v_\sigma) c \pi_\infty(v_\sigma^*) )
&=&
 \langle \pi_\infty(M[H^{1/2}](v_\sigma)) \Xi_\infty,  
J \Phi(c) J \pi_\infty(M[H^{1/2}](v_\sigma)) \Xi_\infty \rangle.
\end{eqnarray*}
By Propostion \ref{proposition: commutant},
$J \Phi(c) J \in J \pi_\infty(\CstarU(V))'' J$ commutes with
$\pi_\infty(M[H^{1/2}](v_\sigma))$.
It follows that
\begin{eqnarray*}
\psi( \pi_\infty(v_\sigma) c \pi_\infty(v_\sigma^*) )
&=&
 \langle  \pi_\infty(M[H^{1/2}](v_\sigma)^* \cdot M[H^{1/2}](v_\sigma)) \Xi_\infty,  
J \Phi(c) J\Xi_\infty \rangle.
\end{eqnarray*}
The multiplication operator $h_\sigma$ acting on $\ell_2(V)$ is identical to $M[H^{1/2}](v_\sigma)^* \cdot M[H^{1/2}](v_\sigma)$.
Combining with $J \Xi_\infty = \Xi_\infty$, we have
\begin{eqnarray*}
\psi( \pi_\infty(v_\sigma) c \pi_\infty(v_\sigma^*) )
=
 \langle  \pi_\infty(h_\sigma) \Xi_\infty,  
J \Phi(c) \Xi_\infty \rangle
=
 \langle  \Phi(c) \Xi_\infty,  
J \pi_\infty (h_\sigma) \Xi_\infty \rangle.
\end{eqnarray*}
Because the operator $h_\sigma$ acting on $\ell_2(V)$ is diagonal,
it satisfies \[\pi_\infty (h_\sigma) \Xi_\infty = \Delta^{1/2} \pi_\infty (h_\sigma) \Xi_\infty.\]
It follows that
\begin{eqnarray*}
\psi( \pi_\infty(v_\sigma) c \pi_\infty(v_\sigma^*) )
&=&
 \langle  \Phi(c) \Xi_\infty,  
J \Delta^{1/2} \pi_\infty (h_\sigma) \Xi_\infty \rangle\\
&=&
 \langle  \Phi(c) \Xi_\infty,  
\pi_\infty (h_\sigma^*) \Xi_\infty \rangle\\
&=&
 \langle \pi_\infty (h_\sigma) \Phi(c) \Xi_\infty,  
 \Xi_\infty \rangle\\
&=&
 \langle \Phi(\pi_\infty (h_\sigma) c) \Xi_\infty,  
 \Xi_\infty \rangle.
\end{eqnarray*}
Thus we conclude that $\psi( \pi_\infty(v_\sigma) c \pi_\infty(v_\sigma^*) )
 = \psi(\pi_\infty (h_\sigma) c)$. 
\end{proof}

\subsection{Almost equivariant $\ell_1$-functions on $\sqcup_n V_n^2$}

We prove in Lemma \ref{lemma: approximation by ell_1} that for a finite family $\Sigma$ of controlled partially defined injections on $\sqcup_n V_n^2$,
there exist a subsequence $(V_{n(k)})_k$ and a controlled sequence of probability measures $(p_{n(k)})_k$ on $V_{n(k)}^2$ which are almost equivariant under the action of $\Sigma$.
We make use of the technique developed in \cite[Lemma 2.4]{ConnesInjectiveFactors}.

Denote by $e$ the rank one projection onto $\C \Xi_\infty \subset \Hil_\infty$.

\begin{lemma}
\label{lemma: good density operators}
There exists a net $(s_\nu)$ of density operators on $\Hil_\infty$ satisfying the following conditions.
\begin{enumerate}
\item
Every
$s_\nu$ is a finite sum of rank one operators of the form $\pi_\infty(a) e \pi_\infty(a)^*$, where $a$ is an operator on $\ell_2(V)$ with finite propagation.
\item
For every $c \in \B(\Hil_\infty)$, $\lim_\nu \mathrm{Tr}(s_\nu c) = \psi(c)$.
\item
For every controlled partially defined injection $\sigma$ on $V$,
\[\lim_\nu \mathrm{Tr}( | s_\nu \pi_\infty(h_\sigma) - \pi_\infty(v_\sigma^*) s_\nu \pi_\infty(v_\sigma) | ) = 0.\]
\end{enumerate}
\end{lemma}

\begin{proof}
In the weak-$*$ topology, 
the state $\psi$ on $\B(\Hil_\infty)$ can be approximated by some net $(t_\nu)_{\nu \in I}$
of density operators on $\Hil_\infty$.
Because $\Xi_\infty$ is a cyclic vector of the algebra $\pi_\infty(\cup_R E_R)$, every density operator on $\Hil_\infty$ can be approximated by finite sum of rank one operator of the form $\pi_\infty(a) e \pi_\infty(a)^*$, where $a \in \cup_R E_R$.
We may assume that $t_\nu$ is such a finite sum.
Let $\Sigma$ be a finite family of controlled partially defined injections on $V$.
For every $\sigma \in \Sigma$, the trace class operators 
\[t_\nu \pi_\infty(h_\sigma) - \pi_\infty(v_\sigma^*) t_\nu \pi_\infty(v_\sigma) \]
converges to $0$, in the weak topology.
Indeed, by Lemma \ref{lemma: equivariant state}, we have
\begin{eqnarray*}
&&
\lim_\nu \mathrm{Tr}(\left\{t_\nu \pi_\infty(h_\sigma) - \pi_\infty(v_\sigma^*) t_\nu \pi_\infty(v_\sigma)\right\} c) \\
&=&
\lim_\nu \mathrm{Tr}(t_\nu \pi_\infty(h_\sigma) c)
-
\lim_\nu \mathrm{Tr}(t_\nu \pi_\infty(v_\sigma) c \pi_\infty(v_\sigma)^*)\\
&=&
\psi(\pi_\infty(h_\sigma) c) - \psi(\pi_\infty(v_\sigma) c \pi_\infty(v_\sigma)^*) = 0,
\quad c \in \B(\Hil_\infty).
\end{eqnarray*}
For every $\nu \in I$, the weak closure of the set of convex combinations
\[\mathrm{conv} \{(t_\mu \pi_\infty(h_\sigma) - \pi_\infty(v_\sigma^*) t_\mu \pi_\infty(v_\sigma))_\sigma \ |\ \mu \in I, \nu \le \mu\} \subset \bigoplus_{\sigma \in \Sigma} \B(\Hil_\infty)_*\]
contains $(0)_{\sigma \in \Sigma}$.
By the Hahn--Banach theorem, the weak closure of the convex set coincides with its norm closure.
It follows that for every $\nu \in I$, there exists a density operator
$s_\nu \in \mathrm{conv} \{t_\mu \ |\ \mu \in I, \nu \le \mu\}$
such that the net $(s_\nu)_{\nu \in I}$ converges to $\psi$ in the weak-$*$ topology of $\B(\Hil_\infty)^*$ and that \[s_\nu \pi_\infty(h_\sigma) - \pi_\infty(v_\sigma^*) s_\nu \pi_\infty(v_\sigma) \]
converges to $0$, in the trace $1$-norm for every $\sigma \in \Sigma$.
\end{proof}

\begin{lemma}
\label{lemma: description of s}
Let $s$ be a finite sum of rank one operators of the form $\pi_\infty(a) e \pi_\infty(a)^*$, where $a \in \cup_R E_R$.
Then there exists a positive element $\eta$ of $\ell_\infty(\sqcup_n V_n^2)$
satisfying the following conditions:
\begin{enumerate}
\item
The support of $\eta$ is controlled.
\item
For every finite family $\mathcal{F}$ of elements of $\ell_\infty(\sqcup_n V_n^2)$,
there exists an increasing sequence $n(k)$ of natural numbers such that
for every $\zeta \in \mathcal{F}$,
\[\lim_{k \to \infty} \sum_{x, y \in V_{n(k)}} \zeta(x, y) \eta(x, y) \xi_{n(k)}(y)^2 = \mathrm{Tr}(s \varpi_\infty(\zeta)).\]
\end{enumerate}
\end{lemma}

\begin{proof}
Suppose that $s$ is of the form
$s = \sum_{l = 1}^m \pi_\infty(a^{(l)}) e \pi_\infty(a^{(l)})^*$.
Describe $a^{(l)}$ by
$a^{(l)} = \left( \left[ a^{(l)}_{n, x, y} \right]_{x, y \in V_n} \right)_{n \in \N}$.
Define $\eta \in \ell_\infty (V)$ by
$\eta(x, y) = \sum_{l = 1}^m \left| a^{(l)}_{n, x, y} \right|^2, (x, y \in V_n^2)$.
The support of $\eta$ is controlled, since the propagation of $a^{(l)}$ is finite.
For every $\zeta \in \ell_\infty \left( \sqcup_n V_n^2 \right)$, 
$\mathrm{Tr}(s \varpi_\infty(\zeta))$ is equal to
\begin{eqnarray*}
\sum_{l = 1}^m \left\langle \varpi_\infty(\zeta) \pi_\infty 
\left( a^{(l)} \right) \Xi_\infty,  \pi_\infty \left( a^{(l)} \right) \Xi_\infty \right\rangle
=
\sum_{l = 1}^m \phi_\infty \left( a^{(l)*} M[\zeta] \left( a^{(l)} \right) \right).
\end{eqnarray*}
Let $\mathcal{F}$ be an arbitrary finite family of elements of $\ell_\infty \left( \sqcup_n V_n^2 \right)$.
By the definition of $\phi_\infty$, there exists a subsequence $(\phi_{n(k)})$ satisfying that
for every $\zeta \in \mathcal{F}$,
\[
\mathrm{Tr}(s \varpi_\infty(\zeta)) =
\lim_{k \to \infty} \sum_{l = 1}^m \phi_{n(k)} \left( a^{(l)*} M[\zeta] \left( a^{(l)} \right) \right).\]
Recalling the definition of $\phi_{n(k)}$,
we have
\begin{eqnarray*}
\sum_{l = 1}^m \phi_{n(k)} \left( a^{(l)*} M[\zeta] \left( a^{(l)} \right) \right)
&=&
\sum_{l = 1}^m \sum_{y \in V_{n(k)}} \left\langle M[\zeta]\left( a^{(l)} \right) \xi_{n(k)}(y) \delta_y, 
a^{(l)} \xi_{n(k)}(y) \delta_y \right\rangle\\
&=&
\sum_{l = 1}^m \sum_{y \in V_{n(k)}} \sum_{x \in V_{n(k)}} \zeta(x, y) 
\left| a^{(l)}_{n(k), x, y} \right|^2 \xi_{n(k)}(y)^2\\
&=&
\sum_{y \in V_{n(k)}} \sum_{x \in V_{n(k)}} \zeta(x, y) \eta(x, y) \xi_{n(k)}(y)^2.
\end{eqnarray*}
\end{proof}

\begin{lemma}
\label{lemma: approximation by ell_1}
Let $\Sigma$ be a finite family of controlled partially defined injections on $V$.
Let $F$ be a finite family of elements in $\ell_\infty (V)$.
Let $\epsilon$ be a positive number.
For arbitrary $\Sigma$, $F$ and $\epsilon$, there exist 
\begin{itemize}
\item
a positive number $S$,
\item
an increasing sequence $(n(k))_k$ of natural numbers,
\item
and probability measures $p_{n(k)} \in \ell_1 \left( V_{n(k)}^2 \right)_+$ on $V_{n(k)}^2$,
\end{itemize}
satisfying the following conditions:
\begin{enumerate}
\item
For every $k \in \N$, and for every $x, y \in V_{n(k)}$, if $d_{n(k)}(x, y) > S$, then $p_{n(k)} (x, y) = 0$.
\item
For every $f \in F$, and for every $k \in \N$,
\[
\left| \sum_{x \in V_{n(k)}}
\sum_{y \in V_{n(k)}}
f(x) p_{n(k)}(x, y) - \phi_\infty(f) \right| < \epsilon.
\]
\item
For every $\sigma = (\sigma_n)_n$ in $\Sigma$,
and for every $k \in \N$,
\[
\left\| 
(\sigma_n \times \mathrm{id})^* \left( p_{n(k)} \cdot \chi \left( \mathrm{image} \left( \sigma_{n(k)} \right) \times V_{n(k)} \right) \right) - p_{n(k)} \cdot \iota(h_\sigma) 
\right\|_1 
< \epsilon.\]
Here, $\chi(\cdot)$ stands for the definition function, and  $\|\cdot\|_1$ stands for sum of absolute values of entries.
\end{enumerate}
\end{lemma}
\begin{proof}
By Lemma \ref{lemma: good density operators},
there exists a density operator $s$ such that
\begin{enumerate}
\item[(i)]
$s$ is a finite sum of $\pi_\infty(a) e \pi_\infty(a)^*$, where $a \in \cup_R E_R$;
\item[(ii)]
for every $f \in F$, 
\[
|\mathrm{Tr}(s \varpi_\infty(\iota(f))) - \phi_\infty(f)| 
=
|\mathrm{Tr}(s \pi_\infty(f)) - \psi(\pi_\infty(f))| 
< \epsilon;\]
\item[(iii)]
for every $\sigma \in \Sigma$,
$\mathrm{Tr} (|s \pi_\infty(h_\sigma) - \pi_\infty(v_\sigma^*) s \pi_\infty(v_\sigma) |) < \epsilon$.
\end{enumerate}
By Lemma \ref{lemma: description of s},
there exists a positive element $\eta$ of $\ell_\infty(V)$
satisfying the following conditions:
\begin{enumerate}
\item[(a)]
The support of $\eta$ is controlled.
\item[(b)]
\label{item: subsequence}
For every finite family $\mathcal{F}$ of elements of $\ell_\infty(\sqcup_n V_n^2)$,
there exists an increasing sequence $n(k)$ of natural numbers such that
for every $\zeta \in \mathcal{F}$,
\[\lim_{k \to \infty} \sum_{x, y \in V_{n(k)}} \zeta(x, y) \eta(x, y) \xi_{n(k)}(y)^2 = \mathrm{Tr}(s \varpi_\infty(\zeta)).\]
\end{enumerate}
We define a sequence of positive functions $p_n \in \ell_1(V_n^2)_+$ by $p_n (x, y) = \eta(x, y) \xi_n(y)^2$.
Since the support of $\eta$ is controlled, the union of the supports of $(p_n)_n$ is also controlled.

We define $\mathcal{F}$ by the following collections:
\begin{itemize}
\item
The constant function $1$ on $\sqcup_n V_n^2$.
\item
Functions of the form $\iota(f)$, $f \in F$.
\item
For $\sigma = (\sigma_n) \in \Sigma$, choose
a function $\zeta_\sigma \in \ell_\infty(\sqcup_n \mathrm{dom}(\sigma_n) \times V_n)$ such that
$\|\zeta_\sigma\|_\infty = 1$
and that for every $n \in \N$,
\begin{eqnarray*}
&&
\| (\sigma_n \times \mathrm{id})^* (p_n \cdot \chi(\mathrm{image}(\sigma_n) \times V_n)) - p_n \cdot \iota(h_\sigma) \|_1\\
&=&
\sum_{x \in \mathrm{dom}(\sigma_n)} \sum_{y \in V_n} 
\left( 
p_n(\sigma_n(x), y) 
- 
p_n (x, y)  h_\sigma (x) \right) \zeta_\sigma(x, y).
\end{eqnarray*}
Add $(\sigma \times \mathrm{id})_* (\zeta_\sigma)$ and 
$\zeta_\sigma \cdot \iota(h_\sigma)$ to $\mathcal{F}$.
\end{itemize}
For this collection $\mathcal{F}$, choose an increasing sequence $n(k)$ of natural numbers satisfying the equation in condition (b).
By $1 \in \mathcal{F}$, and by condition (b),
we have
\begin{eqnarray*}
\lim_{k \to \infty}
\| p_{n(k)} \|_1
=
\lim_{k \to \infty} \sum_{x, y \in V_{n(k)}} \eta(x, y) \xi_{n(k)}(y)^2
= \mathrm{Tr}(s \varpi_\infty(1)) = 1.
\end{eqnarray*}
For $f \in F$, we have
\begin{eqnarray*}
&&
\lim_{k \to \infty}
\left| \sum_{x \in V_{n(k)}}
\sum_{y \in V_{n(k)}}
f(x) p_{n(k)} (x, y) - \phi_\infty(f) \right|\\
&=&
\lim_{k \to \infty}
\left| \sum_{x \in V_{n(k)}}
\sum_{y \in V_{n(k)}}
f(x) \eta(x, y) \xi_{n(k)}(y)^2 - \phi_\infty(f) \right|\\
&=&
\left| \mathrm{Tr}(s \varpi_\infty(\iota(f))) - \phi_\infty(f) \right|
<
\epsilon.
\end{eqnarray*}
At the inequality in the last line, we used condition (ii).
For every $\sigma \in \Sigma$, we have
\begin{eqnarray*}
&&
\lim_{k \to \infty}
\| (\sigma_{n(k)} \times \mathrm{id})^* (p_{n(k)} \cdot \chi(\mathrm{image}(\sigma_{n(k)}) \times V_{n(k)})) - p_{n(k)} \cdot \iota(h_\sigma) \|_1\\
&=&
\lim_{k \to \infty}
\sum_{x \in \mathrm{dom}(\sigma_{n(k)})} \sum_{y \in V_{n(k)}} 
\left( 
p_{n(k)}(\sigma_{n(k)}(x), y) 
- 
p_{n(k)} (x, y)  h_\sigma (x) \right) \zeta_\sigma(x, y)\\
&=&
\lim_{k \to \infty}
\sum_{x \in \mathrm{image}(\sigma_{n(k)})} \sum_{y \in V_{n(k)}} 
p_{n(k)}(x, y) \zeta_\sigma \left( \sigma_{n(k)}^{-1}(x), y \right) \\
&& \hspace{20mm} - 
\lim_{k \to \infty}
\sum_{x \in \mathrm{dom}(\sigma_{n(k)})} \sum_{y \in V_{n(k)}} 
p_{n(k)} (x, y)  h_\sigma (x) \zeta_\sigma(x, y)\\
&=&
\mathrm{Tr} (s \varpi_\infty((\sigma \times \mathrm{id})_* (\zeta_\sigma)) )
-
\mathrm{Tr} (s \varpi_\infty(\iota(h_\sigma) \cdot \zeta_\sigma)) ).
\end{eqnarray*}
At the equality in the last line, we used condition (b).
By Lemma \ref{lemma: normalizer}, we obtain
\begin{eqnarray*}
&&
\lim_{k \to \infty}
\| (\sigma_{n(k)} \times \mathrm{id})^* (p_{n(k)} \cdot \chi(\mathrm{image}(\sigma_{n(k)}) \times V_{n(k)})) - p_{n(k)} \cdot \iota(h_\sigma) \|_1\\
&=&
\mathrm{Tr} (s \pi_\infty(v_\sigma) \varpi_\infty(\zeta_\sigma) \pi_\infty(v_\sigma)^* )
-
\mathrm{Tr} (s \pi_\infty(h_\sigma) \varpi_\infty(\zeta_\sigma))\\
&\le&
\mathrm{Tr} (| \pi_\infty(v_\sigma)^* s \pi_\infty(v_\sigma) - s \pi_\infty(h_\sigma)|) \cdot \|\varpi_\infty(\zeta_\sigma)\|
<
\epsilon.
\end{eqnarray*}
At the last inequality, we used condition (iii). Therefore, if $k$ is large, then the sequence of probability measures $p_{n(k)} / \| p_{n(k)} \|_1$ satisfies the conditions in the lemma.
\end{proof}

\subsection{Another state $\omega_\infty$ of $\CstarU(V)$}
Let $\xi_n \in \ell_2(V_n)$ be the sequence of positive unit vectors satisfying the conditions in Subsection \ref{subsection: diagonal state at infinity}.
We have already studied a sequence of states $(\phi_n)_n$ of $\CstarU(V) \otimes_\mathrm{min} \CstarU(V)^\op$ given by the vectors
\[\Xi_n = \sum_{x \in V_n} \xi_n(x) \delta_{(x, x)} \in \ell_2(V_n^2).\]
Here, we consider another sequence of states $\omega_n = \langle \ \cdot \ \xi_n, \xi_n \rangle$ on $\CstarU(V)$ given by the vectors $\xi_n \in \ell_2(V_n)$.
The restriction of $\phi_n$ on the first tensor factor $\CstarU(V)$ and $\omega_n$ are {\it not} identical, but these restrictions on the diagonal subalgebra $\ell_\infty(V)$ are identical.
We can choose an accumulation point $\phi_\infty$ of $(\phi_n)_n \subset (\CstarU(V) \otimes_\mathrm{min} \CstarU(V)^\op)^*$ and an accumulation point $\omega_\infty$ of $(\omega_n)_n \subset \CstarU(V)^*$ such that these restrictions 
$\phi_\infty |_{\ell_\infty(V) \otimes \C}$ and 
$\omega_\infty |_{\ell_\infty(V)}$ on $\ell_\infty (V)$ are identical.
The following lemma implies that $\omega_\infty |_{\ell_\infty(V)}$ determines $\omega_\infty$.

\begin{lemma}
For every controlled partially defined injection $\sigma = (\sigma_n)_n$ on $V = \sqcup_n V_n$,
and for every $f \in \ell_\infty(V)$,
we have
$\omega_\infty(v_\sigma f) = \omega_\infty \left( \sqrt{h_\sigma} f \right)$.
\end{lemma}

\begin{proof}
For every $n \in \N$,
we have
\begin{eqnarray*}
\omega_n(v_\sigma f) = 
\sum_{x \in \mathrm{dom}(\sigma_n)} \xi_{\sigma_n(x)} \xi_x f(x)
=
\sum_{x \in \mathrm{dom}(\sigma_n)} \xi_x^2 \sqrt{h_\sigma(x)} f(x)
=
\omega_n \left( \sqrt{h_\sigma} f \right).
\end{eqnarray*}
Therefore we have $\omega_\infty(v_\sigma f) = \omega_\infty \left( \sqrt{h_\sigma} f \right)$.
\end{proof}

\subsection{Localization property of the sequence of states $(\omega_n)_n$ on $\CstarU(V)$}

For every $n \in \N$, and positive number $S$, we denote by $\Theta_{n, S}$ the unital completely positive map
\[\Theta_{n, S} \colon \B(\ell_2(V_n)) \to \prod_{y \in V_n} \B(\ell_2(N(y; S)))\]
given by compression $\Theta_{n, S}(a) = (P_{N(y; S)} a P_{N(y; S)})_{y \in V_n}$.

\begin{theorem}
\label{theorem: approximation}
Let $(G_n = (V_n, {\rm Edge}_n))_n$ be a sequence of finite connected graphs.
Suppose that $V = \sqcup_n V_n$ has bounded geometry.
Let $(\xi_n)_n$ be a sequence of unit vectors in $\ell_2(V_n)$.
Let $\mathcal{G}$ be a finite family of elements in the uniform Roe algebra $\CstarU(V)$.
Let $\epsilon$ be a positive number.
If $\CstarU(V)$ is locally reflexive,
then
for arbitrary $(\xi_n)_n$, $\mathcal{G}$, and $\epsilon$,
there exists a positive number $S$, an increasing sequence of natural numbers $(n(k))_k$, and states $\widehat{\omega}_{n(k)}$ on $\prod_{y \in V_{n(k)}} \B(\ell_2(N(y; S)))$
such that
\[\left| \widehat{\omega}_{n(k)} (\Theta_{n(k), S}(a)) - \langle a \xi_{n(k)}, \xi_{n(k)} \rangle \right| < \epsilon, \quad a \in \mathcal{G}.\]
\end{theorem}
In this theorem, we do not need any assumption on the unit vectors $\xi_n \in \ell_2(V_n)$.

\begin{proof}
We first reduce our task as follows.
\begin{enumerate}
\item
We may assume that all the elements of $\mathcal{G}$ have finite propagation.
\item
Adding edges to $G_n$,
we may further assume that the propagation of elements of $\mathcal{G}$ is at most $1$.
\item
There exist a finite family $G$ of $\ell_\infty(V)$ and a finite family $\Sigma$ of partially defined injections on $V$ such that for every $\sigma \in \Sigma$, the propagation of $v_\sigma$ is at most $1$ and that the linear span of
$\{v_\sigma f \ |\ \sigma \in \Sigma, f \in G\}$
contains $\mathcal{G}$.
We replace $\mathcal{G}$ with $\{v_\sigma f \ |\ \sigma \in \Sigma, f \in G\}$.
\item
We may assume that
there exists a positive number $\Lambda$ such that for every $n \in \N$ and 
for every $(x, y) \in {\rm Edge}_n$,
\[|\xi_n(y)|^2 \le \Lambda |\xi_n(x)|^2.\]
\item
We may assume that
for every $n$, the entries of $\xi_n$ are positive.
\item
We may assume that for $v_\sigma f \in \mathcal{G}$, the sequence $(\langle v_\sigma f \xi_n, \xi_n \rangle)_n$ converges.
\end{enumerate}
We prove that the above reduction processes are possible as follows.
\begin{enumerate}
\item
All the elements of $\mathcal{G}$ can be approximated by some operators with finite propagation.
\item
For every propagation $1 \le R$,
define new edge sets ${\rm Edge}_n^{(R)}$ by
\[{\rm Edge}_n^{(R)} = \{(x, y) \in V_n \ |\ d_n(x, y) \le R\}.\]
The sequence of graphs $\left( \left( V_n, {\rm Edge}_n^{(R)} \right) \right)_n$ also has bounded geometry.
The uniform Roe algebras of these sequences of vertex sets are identical.
\item
Every controlled subset of $\sqcup_n V_n$ is a finite union of graphs of controlled partially defined injections, by \cite[Lemma 4.10]{RoeLectureNote}.
\item
It suffices to show that there exists a subset $W_n$ of $\mathrm{supp}(\xi_n)$ such that
$\xi_n |_{W_n}$ is uniformly close to $\xi_n$ in $\ell_2$-norm and that the derivatives 
$|\xi_n (y)| / |\xi_n (x)|$ defined on $(x, y) \in \mathrm{Edge}_n^{(R)}$ are uniformly bounded.

Let $(V, \mathrm{Edge})$ be a finite connected graph. Let $\xi$ be a unit vector in $\ell_2(V)$ whose support is $V$.
Let $\epsilon$ be an arbitrary positive number.
Define $M$ by
\[\sup_{x \in V} \sharp \{y \in V \ |\ (x, y) \in \mathrm{Edge} \}.\]
Let $\Lambda$ be a positive constant larger than $M$ and $1$.
The constant $\Lambda$ will be determined by $\epsilon$ and $M$ later.
Define a relation $\ll$ on $V$ as follows:
\[x \ll y, \textrm{\ if\ } (x, y) \in \mathrm{Edge} \textrm{\ and\ }\Lambda |\xi(x)|^2 \le |\xi(y)|^2.\]
Define a sequence of disjoint subsets $(U(l))_l$ of $V$ as follows:
\begin{eqnarray*}
U(0) &=& \{y \in V \ |\ \exists x \in V, x \ll y, \lnot\exists z \in V, y \ll z\},\\
U(1) &=& \{y \in V \ |\ \exists z \in U(0), y \ll z\},\\
U(l + 1) &=& \{y \in V \ |\ \exists z \in U(l), y \ll z\} \setminus \cup_{k = 1}^l U(k).
\end{eqnarray*}
By the definitions of $\ll$ and $M$, we have
$\|\xi |_{U(l + 1)}\|^2 \le \frac{M}{\Lambda} \|\xi |_{U(l)}\|^2$.
It follows that
$\sum_{l = 1}^\infty \|\xi |_{U(l)}\|^2 \le \sum_{l = 1}^\infty \frac{M^l}{\Lambda^l} \|\xi |_{U(0)}\|^2
\le \frac{M}{\Lambda - M}$.
Define $\Lambda$ by $(1 + 1 / \epsilon) M$.
Define $W$ by $V \setminus \left(\cup_{l = 1}^\infty U(l)\right)$.
Thus we have
$\|\xi - \xi |_W \|^2 \le \sum_{l = 1}^\infty \|\xi |_{U(l)}\|^2 \le \epsilon$.
If $x, y \in W$, then $x \ll y$ does not hold.
This means that $|\xi(y)|^2 \le \Lambda |\xi(x)|^2$.

Apply the above procedure to $\mathrm{supp} (\xi_n)$. Replace $\mathrm{supp} (\xi_n)$ with $W_n$. 
We can use common $\Lambda$ for every $n$.
Replace $\xi$ with $\xi |_{W_n}$.
Replace $\mathcal{G} = \{v_\sigma f \ |\ \sigma \in \Sigma, f \in G\}$ with their restrictions on the new Hilbert space $\ell_2(\sqcup_n W_n)$.
The propagation of elements of the new finite family $\mathcal{G}$ is $1$.
\item
Multiplying a unitary operator $w$ in $\ell_\infty V$, we obtain a sequence of positive vectors $\xi_n \in \ell_2(V_n)$.
We also need to replace $\mathcal{G} = \{v_\sigma f \ |\ \sigma \in \Sigma, f \in G\}$ with their conjugations by $w v_\sigma f w^* = v_\sigma \widetilde{f}$.
\item
Replace $(V_n)_n$ with its subsequence such that $\lim_n \langle v_\sigma f \xi_n, \xi_n \rangle$ exists for every $v_\sigma f \in \mathcal{G}$.
\end{enumerate}

For the rest of the proof, we assume the above six items.
We make use of states $\phi_\infty$ and $\omega_\infty$ in the previous subsections.
Define $F \subset \ell_\infty(V)$ by 
$\left\{ \left. \sqrt{h_\sigma} f \ \right| \ \sigma \in \Sigma, f \in G \right\}$.
For $F$, $\Sigma$, and arbitrary $\epsilon > 0$, there exist 
\begin{itemize}
\item
a positive number $S$,
\item
an increasing sequence $(n(k))_k$ of natural numbers,
\item
and probability measures $p_{n(k)} \in \ell_1 \left( V_{n(k)}^2 \right)_+$ on $V_{n(k)}^2$,
\end{itemize}
satisfying the conditions in Lemma \ref{lemma: approximation by ell_1}.

We first claim that for every $k$, and for every $\sigma = (\sigma_n)_n \in \Sigma$,
\begin{eqnarray*}
\label{quantity}
\sum_{x \in \mathrm{dom}(\sigma_{n(k)})} \sum_{y \in V_{n(k)}}
\left|
\sqrt{h_\sigma(x)} p_{n(k)}(x, y) - \sqrt{p_{n(k)}(x, y)} \sqrt{p_{n(k)}(\sigma_{n(k)}(x), y)}
\right|
\end{eqnarray*}
is small. 
By condition $(4)$, $h_\sigma$ is at least $\Lambda^{-1}$.
The above quantity is 
equal to
\begin{eqnarray*}
\sum_{x, y}
\left|
\sqrt{h_\sigma(x)} \sqrt{p_{n(k)}(x, y)} - \sqrt{p_{n(k)}(\sigma_{n(k)}(x), y)}
\right| \sqrt{p_{n(k)}(x, y)}
\end{eqnarray*}
and
bounded by
\begin{eqnarray*}
&&
\sqrt{\Lambda}
\sum_{x, y}
\left|
\sqrt{h_\sigma(x)} \sqrt{p_{n(k)}(x, y)} - \sqrt{p_{n(k)}(\sigma_{n(k)}(x), y)}
\right| \sqrt{h_\sigma(x)} \sqrt{p_{n(k)}(x, y)} \\
&\le&
\sqrt{\Lambda}
\sum_{x, y}
\left|
\sqrt{h_\sigma(x)} \sqrt{p_{n(k)}(x, y)} - \sqrt{p_{n(k)}(\sigma_{n(k)}(x), y)}
\right| \\
&& \hspace{20mm} \times
\left(
\sqrt{h_\sigma(x)} \sqrt{p_{n(k)}(x, y)} + \sqrt{p_{n(k)}(\sigma_{n(k)}(x), y)}
\right) \\
&=&
\sqrt{\Lambda}
\sum_{x, y}
\left|
h_\sigma(x) p_{n(k)}(x, y) - p_{n(k)}(\sigma_{n(k)}(x), y)
\right|
< \sqrt{\Lambda} \epsilon.
\end{eqnarray*}
The equality in the last line is due to the last condition in Lemma \ref{lemma: approximation by ell_1}.

The limit $\lim_k \langle v_\sigma f \xi_{n(k)}, \xi_{n(k)} \rangle$ satisfies the following:
\begin{eqnarray*}
&&
\lim_k \langle v_\sigma f \xi_{n(k)}, \xi_{n(k)} \rangle
= \omega_\infty(v_\sigma f)
= \omega_\infty \left( \sqrt{h_\sigma} f \right)
= \phi_\infty \left( \sqrt{h_\sigma} f \right)
\\
&\sim_\epsilon&
\sum_{x \in \mathrm{dom}(\sigma_{n(k)})} \sum_{y \in V_{n(k)}}
\sqrt{h_\sigma(x)} f(x) p_{n(k)}(x, y)\\
&\sim_{\sqrt{\Lambda} \epsilon \|f\|_\infty}&
\sum_{x \in \mathrm{dom}(\sigma_{n(k)})} \sum_{y \in V_{n(k)}}
f(x) \sqrt{p_{n(k)}(x, y)} \sqrt{p_{n(k)}(\sigma_{n(k)}(x), y)}.
\end{eqnarray*}
Here $\sim_\square$ means that the difference between the left hand side and the right hand side is at most $\square$.
Define a unit vector $(q_y)_y$ in $\bigoplus_{y \in V_{n(k)}} \ell_2(N(y, S))$ by
$q_y = \sum_x \sqrt{p_{n(k)}(x, y)} \delta_x$.
For every $x, y$, we have
\begin{eqnarray*}
f(x) \sqrt{p_{n(k)}(x, y)} \sqrt{p_{n(k)}(\sigma_{n(k)}(x), y)}
=
\left\langle v_\sigma f \sqrt{p_{n(k)}(x, y)} \delta_x, q_y \right\rangle.
\end{eqnarray*}
For every $y$, we have
\begin{eqnarray*}
\sum_x
f(x) \sqrt{p_{n(k)}(x, y)} \sqrt{p_{n(k)}(\sigma_{n(k)}(x), y)}
=
\left\langle v_\sigma f q_y, q_y \right\rangle.
\end{eqnarray*}
Therefore we have
\[\lim_k \langle v_\sigma f \xi_{n(k)}, \xi_{n(k)} \rangle \quad
\sim_{\sqrt{\Lambda} \epsilon \|f\|_\infty + \epsilon}
\quad
\sum_{y \in V_{n(k)}} \left\langle v_\sigma f q_y, q_y \right\rangle.
\]
Let $\widehat{\omega}_{n(k)}$ be the vector state on $\prod_{y \in V_{n(k)}} \B(\ell_2(N(y; S)))$ given by the unit vector $(q_y)_y \in \bigoplus_{y \in V_{n(k)}} \ell_2(N(y, S))$.
The right hand side $\sum_{y \in V_{n(k)}} \left\langle v_\sigma f q_y, q_y \right\rangle$
is equal to $\widehat{\omega}_{n(k), S}( \Theta_{n(k)}(v_\sigma f))$.
\end{proof}

\section{Main theorem}

\begin{theorem}\label{theorem: main}
For a metric space $X$ with bounded geometry,
if the uniform Roe algebra $\CstarU(X)$ is locally reflexive, then $X$ has the operator norm localization property.
\end{theorem}

As a consequence, we obtain the main theorem in Section 1.

\begin{proof}
Suppose that $X$ does not have the operator norm localization property
and that $\CstarU(X)$ is locally reflexive.
By Lemma \ref{lemma: not ONLP},
there exist:
\begin{itemize}
\item
a sequence of disjoint subsets $V_n$ of $X$;
\item
a sequence of positive matrices $b_n$ acting on $\ell_2(V_n)$ with norm $1$;
\item
a sequence $S_n$ of positive numbers
\end{itemize}
satisfying that:
\begin{enumerate}
\item\label{item: not localized}
for every $n \in \N$ and $Y \subset V_n$, if $\mathrm{diam}(Y) \le S_n$, 
then $\| b_n |_{\ell_2(Y)} \| < 1/ 3$;
\item
$\lim_n S_n = \infty$;
\item
the propagation of $b := \sum_n b_n$ is finite.
\end{enumerate}
Denote by $R$ the propagation of $b$.

Define a sequence of graphs $\left( G_n^{(R)} = \left( V_n, \mathrm{Edge}_n^{(R)} \right) \right)$ by
\[\mathrm{Edge}_n^{(R)} = \{(x, y) \in V_n \ |\ d(x, y) \le R \}.\]
In the case that $G_n^{(R)}$ is not connected, choose a connected component $W_n \subset V_n$, such that the norm of $b_n P_{W_n}$ is $1$. 
Replace $V_n$ with $W_n$. Replace $b$ with $b_n P_{W_n}$.
Denote by $d_n$ the graph metric on $V_n$ defined by $\mathrm{Edge}_n^{(R)}$.
Since $d | _{V_n \times V_n} \le R d_n$,
the uniform Roe algebra $\CstarU(\sqcup_n (V_n, d_n))$ is a non-unital subalgebra of
$\CstarU(X)$.
Therefore $\CstarU(\sqcup_n (V_n, d_n))$ is also locally reflexive.
Note that $b$ is an element of the subalgebra.

Since $b_n$ are positive semidefinite matrices with norm $1$,
there exists a unit vector $\xi_n \in \ell_2(V_n)$ satisfying that $b_n \xi_n = \xi_n$.
By Theorem \ref{theorem: approximation},
there exist a positive constant $S$, infinitely many natural numbers $n$, and states $\widehat{\omega}_n$ on $\prod_{y \in V_n} \B(\ell_2(N(y; S)))$ such that
\[\quad
\left| \widehat{\omega}_n (\Theta_{n, S}(b)) - \langle b \xi_n, \xi_n \rangle \right| < 1/3.\]
Choosing large $n$, we also have $S < S_n$.
Since $\langle b \xi_n, \xi_n \rangle = 1$, we have
\begin{eqnarray*}
\left\| \Theta_{n, S}(b) \right\| \ge \left| \widehat{\omega}_n (\Theta_{n, S}(b)) \right| > 2/3.
\end{eqnarray*}
This inequality contradicts condition (\ref{item: not localized}).
\end{proof}

\bibliographystyle{amsalpha}
\bibliography{RoeAlg.bib}

\end{document}